\documentclass{amsart}

\usepackage{amscd, amsmath}
\usepackage{hyperref}
\usepackage{amsmath,amsfonts,amsthm,amssymb,graphics}
\usepackage{paralist}
\usepackage{amssymb,latexsym}
\usepackage{amsmath}

\newcommand{\g}{\geqslant}
\newcommand{\ar}{\rangle}
\newcommand{\al}{\langle}

\newcommand{\RR}{\mathbb{R}}

\newcommand{\pr}{\mathcal{M}}

\newcommand{\p}{\partial}

\newcommand{\les}{\leqslant}
\newcommand{\lesa}{\lesssim}

\theoremstyle{plain}
\newtheorem{theorem}{Theorem}

\newtheorem{lemma}[theorem]{Lemma}

\theoremstyle{remark}
\newtheorem{remark}{Remark}

\setdefaultenum{(i)}{(a)}{1}{A}
\title[Local well-posedness in Lorenz gauge]{Local well-posedness for the Space-Time Monopole equation in Lorenz Gauge}
\author[N. Bournaveas]{Nikolaos Bournaveas}
\address{
School of Mathematics\\
University of Edinburgh\\
Edinburgh EH9 3JZ\\
United Kingdom}

\email{N.Bournaveas@ed.ac.uk, T.L.Candy@sms.ed.ac.uk}

\author[T. Candy]{Timothy Candy}

\thanks{Article has been published in NoDEA, the final publication is available at http://www.springerlink.com/content/822235v0l315v544/}

\date{\today}
\dateposted{}

\begin{document}
\maketitle

\begin{abstract}
It is known from the work of Czubak \cite{Czubak2010} that the space-time Monopole equation is locally well-posed
in the Coulomb gauge for small initial data in $H^s(\mathbb{R}^2)$ for $s>\frac{1}{4}$. Here we prove
local well-posedness for arbitrary initial data in $H^s(\mathbb{R}^2)$ with $s>\frac{1}{4}$
 in the Lorenz gauge.
\end{abstract}

\section{Introduction}

The space-time Monopole equation is
        \begin{equation}\label{mono unmod}
                F_A = \ast D_A \phi
        \end{equation}
where $F_A$ is the curvature of a one-form connection $A = A_\alpha dx^\alpha$, $D_A$ is a covariant derivative of the Higgs field $\phi$, and $\ast$ is the Hodge star operator with respect to the Minkowski metric diag(-1, 1, 1) on $\RR^{1+2}$. The components of the connection $A=A_\alpha dx^\alpha$, and the Higgs field $\phi$, are maps from $\RR^{1+2}$ into $\mathfrak{g}$
        $$ A_\alpha : \RR^{1+2} \rightarrow \mathfrak{g}, \qquad \qquad \phi : \RR^{1+2} \rightarrow \mathfrak{g},$$
where $\mathfrak{g}$ is a Lie algebra with Lie bracket $[\cdot, \cdot]$. For simplicity we will always assume $\mathfrak{g}$ is the Lie algebra of a matrix Lie group such as $SO(n)$ or $SU(n)$. The curvature $F_A$ of the connection $A$, and the covariant derivative $D_A \phi$  are given by
            $$ F_A = \frac{1}{2}\big( \p_\alpha A_\beta - \p_\beta A_\alpha  + [A_\alpha, A_\beta]\Big)dx^\alpha \wedge dx^\beta, \qquad D_A \phi = (\p_\alpha \phi + [A_\alpha, \phi]\big) dx^\alpha.$$

  The space-time Monopole equation is an example of a non-abelian gauge field theory and can be derived by dimensional reduction from the anti-selfdual Yang-Mills equations, see for instance \cite{Dai2006} or \cite{Mason1996}. It was first introduced by Ward in \cite{Ward1989} as a hyperbolic analog of the Bogomolny equations, or magnetic monopole equations, which describe a point source of magnetic charge. The space-time Monopole equation is an example of a completely integrable system and has an equivalent formulation as a Lax pair. The Lax pair formulation of (\ref{mono unmod}), together with the inverse scattering transform, was used by Dai-Terng-Uhlenbeck in \cite{Dai2006} to prove global existence and uniqueness up to a gauge transform from small initial data in $W^{2, 1}(\RR^2)$. The survey \cite{Dai2006} also contained a number of other interesting results related to the space-time Monopole equation.

  In the current article we study the local well-posedness of the initial value problem for the space-time Monopole equation from rough initial data in $H^s(\RR^2)$. We can think of the equation (\ref{mono unmod}) as a system which is roughly of the form\footnote{The exact formulation depends on the choice of gauge, see below.}
        \begin{equation}\label{generic form} \Box u = |\nabla|^{-1} B( \p u , \p u ) \end{equation}
  where $B$ is some bilinear form. It is well known since the seminal paper of Klainerman-Machedon \cite{Klainerman1994}, that to prove optimal well-posedness for nonlinear wave equations of the form (\ref{generic form}),  the bilinear form $B$ must satisfy certain cancelation properties known as null structure. Consequently, the local behavior of the space-time Monopole equation depends crucially on the presence of null structure.

  The space-time Monopole equation (\ref{mono unmod}) is gauge invariant. More precisely if $(A, \phi)$ is a solution to (\ref{mono unmod}) then so is $(gAg^{-1} + gd g^{-1}, g \phi g^{-1})$ where the gauge transform $g : \RR^{1+2} \rightarrow G$ is smooth and compactly supported map into the Lie group $G$. Thus to obtain a wellposed problem we need to specify a choice of gauge. Traditionally, for nonlinear hyperbolic systems with a gauge freedom such as Maxwell-Klein-Gordon or Maxwell-Dirac, the gauge was chosen to satisfy the Coulomb condition
                    $ \p^j A_j  = 0,$
but more recently null structure has been discovered in the Lorenz gauge as well, \cite{Selberg2010}, \cite{D'Ancona2010a}.
  In the Coulomb gauge, the system (\ref{mono unmod}) can be written as a nonlinear system of wave equations for $(A_1, A_2, \phi)$ coupled with a nonlinear elliptic equation for $A_0$. The advantage of this gauge is that usually the estimates for the elliptic component $A_0$ are quite favorable. Recently Czubak, in\footnote{Though the result was obtain earlier in Czubak's PhD thesis \cite{Czubak2008}}   \cite{Czubak2010}, showed that the space-time Monopole equations in the Coulomb gauge are locally wellposed for small initial data in $H^s$ with $s>\frac{1}{4}$. The small data assumption is an artifact of the choice of the Coulomb gauge, as the elliptic estimates for $A_0$ do not involve time and so to close an iteration argument a smallness assumption is needed.

  In the current article we instead consider the Lorenz gauge condition
            $$ \p_\alpha A^\alpha = 0.$$
  With this choice of gauge the space-time Monopole equations can be written as a purely hyperbolic system and the small data assumption is not needed.  Additionally our proof is substantially shorter as we do not have to combine elliptic estimates with hyperbolic estimates, which can often be technically very inconvenient. Our main result is the following.

%
%
%
%
%

\begin{theorem}\label{main thm}
Assume $s>\frac{1}{4}$ and $\phi_0, a \in H^s(\RR^2)$. Then there exists $T=T(\|\phi_0\|_{H^s(\RR^2)}, \|a\|_{H^s(\RR^2)})>0$ such that the space-time Monopole equation (\ref{mono unmod}) coupled with the Lorenz gauge condition
        $$ \partial^\alpha A_\alpha = 0$$
has a solution $(\phi, A) \in C( [-T, T], H^s(\RR^2))$ with $(\phi(0), A(0)) = (\phi_0, a)$. Moreover the solution is unique in some subspace of $C( [-T, T], H^s(\RR^2))$, the solution map depends continuously on the initial data, and any additional regularity persists in time\footnote{More precisely if $\phi_0,a \in H^r(\RR^2)$ for some $r \g s$, then we also have $(\phi, A) \in C([-T, T], H^r(\RR^2))$ with $T$ only depending on $\|\phi_0\|_{H^s(\RR^2)}$ and $\|a\|_{H^s(\RR^2)}$. }.
\end{theorem}

\begin{remark}
    The space-time Monopole equation is invariant under the scaling $ \lambda A(\lambda t, \lambda x)$, $\lambda \phi(\lambda t, \lambda x)$. Thus (\ref{mono unmod}) is $L^2$ critical and so ideally we would like to prove local well-posedness for $s>0$. However the space-time Monopole equation is essentially a system of nonlinear wave equations, and the fact that we are working in $\RR^{1+2}$ means that there is a gap between what scaling predicts, and the regularity possible via standard null form estimates. More precisely, consider the equation
                $$ \Box u = Q $$
    where $Q$ is a combination of the null forms
            $$ Q_{\alpha\beta}(u, v) = \p_\alpha u \p_\beta v - \p_\beta u \p_\alpha v.$$
    Then the scale invariant space is $H^1 \times L^2$, but standard null form estimates only give well-posedness for $(u(0), \p_tu(0) ) \in H^{s} \times H^{s-1}$ for $s>\frac{5}{4}$. Below $\frac{5}{4}$, it can be shown that the first iterate leaves the data space $H^s$,  see \cite{Zhou1997}. Thus in some sense the regularity $H^{\frac{1}{4}}$ in Theorem \ref{main thm} and the work of Czubak \cite{Czubak2010}, is the limit for iterative methods. On the other hand the space-time Monopole has additional structure which is not used in the proof of Theorem \ref{main thm}. Hence it may be possible to remove the restriction $s>\frac{1}{4}$ by exploiting the structure in a different way.
\end{remark}


\subsection*{Notation}

Throughout this paper $C$ denotes a positive constant which can vary from line to line. The notation $a \lesa b $ denotes the inequality $a \les C b$.
 We let $L^p(\RR^n)$ denote the usual Lebesgue space.  Occasionally we write $L^p(\RR^n) = L^p$ when we can do so without causing confusion. This comment also applies to the other function spaces which appear throughout this paper. If $X$ is a metric space and $I\subset \RR $ is an interval,  then  $C(I, X)$ denotes the set of continuous functions from $I$ into $X$. For $s \in \RR$, we define $H^s$ to be the usual Sobolev space defined using the norm
            $$ \| f \|_{H^s (\RR^2)} = \| \Lambda^s f \|_{L^2(\RR^2)}$$
 where $\widehat{\big(\Lambda^s f \big)} (\xi) = (1+ |\xi|^2)^{\frac{s}{2}} \widehat{f}(\xi)$ and $\widehat{f}$ denotes the Fourier transform of $f$. The space-time Fourier transform of a function $\psi(t, x)$ is denoted by $\widetilde{\psi}(\tau, \xi)$.

\section{Preliminaries}

Recall that the Hodge star operator, $\ast$, is defined for  $\omega \in \bigwedge^p(M)$ by
        $$ (\ast \omega)_{\lambda_{p+1}...\lambda_{n}} = \frac{1}{p!}\eta_{\lambda_1...\lambda_{n}} \omega^{\lambda_{1}... \lambda_p}$$
where $(M, g)$ is a pseudo Riemannian manifold, $\eta$ is the volume form with respect to the metric $g$,  and the previous formula is given in some local coordinate system. If we couple the space-time Monopole equation (\ref{mono unmod}) with the Lorenz gauge condition
        $$ \p^\mu A_\mu = 0$$
and write out the resulting system in terms of $\phi$ and the components $A_\alpha$ we obtain
        \begin{align*}
            \p_t \phi  + \p_1 A_2 - \p_2 A_1 &= [A_2, A_1] + [\phi, A_0] \\
            \p_t A_0  - \p_1 A_1 - \p_2 A_2 &= 0 \\
            \p_t A_1  - \p_1 A_0 - \p_2 \phi &= [ A_2, \phi  ] + [A_1, A_0 ] \\
            \p_t A_2 + \p_1 \phi - \p_2 A_0 &= [\phi, A_1] + [A_2 , A_0]. \end{align*}
Define $u, v : \RR^{1+2} \rightarrow \mathfrak{g} \times \mathfrak{g}$ by
    $$  u = \begin{pmatrix} A_0 + A_1 \\ \phi + A_2 \end{pmatrix}\qquad v= \begin{pmatrix} A_0 - A_1 \\ \phi - A_2 \end{pmatrix} .$$
Then since
    $$
        [A_2, A_1] + [\phi, A_0] \pm \big( [\phi, A_1] + [A_2 , A_0]\big) = [ \phi \pm A_2, A_0 \pm A_1] $$
and
    $$ [ A_2, \phi] + [A_1, A_0] = \frac{1}{2} \big( [ A_2 - \phi, A_2 + \phi] + [ A_1 - A_0, A_1 + A_0]\big) $$
we can write the Monopole equation as
\begin{align*}
    \p_t u_1 - \p_1 u_1 - \p_2 u_2 &= \frac{1}{2}\big( u\cdot v - v \cdot u   \big)  \\
    \p_t u_2 + \p_1 u_2 - \p_2 u_1 &= [u_2, u_1] \\
    \p_t v_1 + \p_1 v_1 + \p_2 v_2 &= \frac{1}{2}\big(  v\cdot u - u \cdot v  \big)  \\
    \p_t v_2 - \p_1 v_2 + \p_2 v_1 &= [v_2, v_1]. \end{align*}
Define the matrices
    $$ \alpha_1 = \begin{pmatrix} 1 & 0 \\ 0 & -1 \end{pmatrix}, \qquad \alpha_2 = \begin{pmatrix} 0 & 1 \\ 1 & 0 \end{pmatrix}, \qquad \beta = \begin{pmatrix} 0 & 1 \\ -1 & 0 \end{pmatrix} $$
and let $\alpha = ( \alpha_1, \alpha_2)$. Then we can rewrite the previous equations in the more concise form
       \begin{equation}\label{mod mono} \begin{cases}
        &\p_t u  - \alpha \cdot \nabla u = N(u, v)\\
        &\p_t v  + \alpha \cdot \nabla v = N(v, u)
       \end{cases}
       \end{equation}
where
        $$N( a, b) = \begin{pmatrix}  \frac{1}{2} \big( a \cdot b - b \cdot a \big)   \\ \beta a \cdot a \end{pmatrix} . $$
We can now restate Theorem \ref{main thm} as follows.
\begin{theorem}\label{local}
Assume $s>\frac{1}{4}$ and $f, g \in H^s$. Then there exists $T=T(\|f\|_{H^s}, \|g\|_{H^s})>0$ such that (\ref{mod mono}) has a solution $(u, v) \in C( [-T, T], H^s)$ with $(u(0), v(0)) = (f, g)$. Moreover, the solution is unique in some subspace of $C( [-T, T], H^s)$, the solution map depends continuously on the initial data, and any additional regularity persists in time\footnote{More precisely if $f, g \in H^r$ for some $r \g s$, then we also have $(u, v) \in C([-T, T], H^r)$ with $T$ only depending on $\|f\|_{H^s}$ and $\|g\|_{H^s}$. }.

\end{theorem}

Note that Theorem \ref{main thm} follows immediately from Theorem \ref{local}. To prove Theorem \ref{local} we will first diagonalise the left hand side of (\ref{mod mono}). Define the projections $\mathcal{M}_{\pm}$ to be the operator with Fourier multiplier  $ m_\pm (\xi) = \frac{1}{2} \left( I \pm \frac{1}{|\xi|} \alpha \cdot \xi \right)$
so
        $$ \widehat{\mathcal{M}_{\pm} f }(\xi) = m_\pm (\xi) \widehat{f}(\xi).$$
It is easy to see that
        $$ f = \mathcal{M}_+ f  + \mathcal{M}_-f, \qquad  \alpha\cdot \nabla = i |\nabla| \big( \mathcal{M}_+ - \mathcal{M}_-\big)$$
and $\mathcal{M}^2_\pm = \mathcal{M}_\pm$, $\mathcal{M}_\pm \mathcal{M}_\mp = 0$. Therefore we can rewrite the above as
         \begin{align*}
         \p_t u_\pm  \mp  i|\nabla|  u_\pm &= \mathcal{M}_\pm N(u, v)\\
         \p_t  v_\pm \mp  i|\nabla |  v_\pm &= \mathcal{M}_\mp N(v, u)
        \end{align*}
where $ u_\pm = \mathcal{M}_\pm u$ and $v_\pm = \mathcal{M}_\mp v$. With this formulation we see that, for short times at least,  $u_+$ and $v_+$ should have Fourier support concentrated on the forwards light cone $\{\tau - |\xi|=0\}$, while $u_-$ and $v_-$ should have Fourier support concentrated on the backwards light cone $\{\tau + |\xi| = 0\}$. Thus the natural spaces to iterate in are the spaces $X^{s, b}_{\pm}$ defined by using the norm
            $$\| \psi \|_{X^{s, b}_\pm}  = \big\| \al \tau \mp |\xi| \ar^b \al \xi \ar^s \widehat{\psi}(\tau, \xi) \big\|_{L^2_{\tau, \xi} }.$$
We also let $H^{s, b}$ be the closely related Wave-Sobolev space defined by
            $$ \| \psi \|_{H^{s, b}} = \big\| \al |\tau| - |\xi| \ar^b \al \xi \ar^s \widehat{\psi}(\tau, \xi) \big\|_{L^2_\xi}.$$
We will iterate in the spaces $u_+, v_+ \in X^{s, b}_+$ and $u_-, v_- \in X^{s, b}_-$ for some $\frac{1}{2}<b<1$ to be chosen later.
It is well known that the proof of Theorem \ref{local} reduces to proving the estimates
        \begin{equation}\label{nonlinear est1} \| \pr_\pm N(u, v) \|_{X^{s, b-1+\epsilon}_\pm} \lesa
                \big( \| u_+ \|_{X^{s, b}_+ } + \| u_- \|_{X^{s, b}_- } + \| v_+ \|_{X^{s, b}_+ } + \| v_- \|_{X^{s, b}_- }\big)^2 \end{equation}
and
        \begin{equation}\label{nonlinear est2} \| \pr_\mp N(v, u) \|_{X^{s, b-1+\epsilon}_\pm} \lesa
                    \big( \| u_+ \|_{X^{s, b}_+ } + \| u_- \|_{X^{s, b}_- } + \| v_+ \|_{X^{s, b}_+ } + \| v_- \|_{X^{s, b}_- }\big)^2 \end{equation}
where $\epsilon>0$ is some small constant depending on $s$ and $b>\frac{1}{2}$, see for instance \cite{Selberg2002} or Section 3 in \cite{Bejenaru2006}. Since $\pr_\pm$ is a bounded operator on $H^s$, and $ \big| |\tau| - |\xi| \big| \les \big| \tau \pm |\xi|\big|$, we see that provided $b+\epsilon<1$, the estimates (\ref{nonlinear est1}) and (\ref{nonlinear est2}) follow from
                    $$ \|  N(u, v) \|_{H^{s, b-1+ \epsilon}} \lesa
                \big( \| u_+ \|_{X^{s, b}_+ } + \| u_- \|_{X^{s, b}_- } + \| v_+ \|_{X^{s, b}_+ } + \| v_- \|_{X^{s, b}_- }\big)^2 $$
and
                $$ \|  N(v, u) \|_{H^{s, b-1+\epsilon}} \lesa
                \big( \| u_+ \|_{X^{s, b}_+ } + \| u_- \|_{X^{s, b}_- } + \| v_+ \|_{X^{s, b}_+ } + \| v_- \|_{X^{s, b}_- }\big)^2 .$$
Now recalling that $u= \pr_+ u_+ + \pr_- u_-$, $v = \pr_- v_+ + \pr_+ v_-$, and
       $$N( a, b) = \begin{pmatrix}  \frac{1}{2} \big( a \cdot b - b \cdot a  \big)   \\ \beta a \cdot a \end{pmatrix},  $$
we can reduce this further to just proving the estimates
            $$ \| \pr_{\pm_1} \Psi \cdot \pr_{\pm_2} \Phi \|_{H^{s, b-1+\epsilon}} \lesa \| \Psi \|_{X^{s, b}_{\pm_1}} \| \Phi \|_{X^{s, b}_{\mp_2}},$$
            $$ \| \beta \pr_{\pm_1} \Psi \cdot \pr_{\pm_2} \Phi \|_{H^{s, b-1+\epsilon}} \lesa \| \Psi \|_{X^{s, b}_{\pm_1}} \| \Phi \|_{X^{s, b}_{\pm_2}},$$
and
            $$ \| \beta \pr_{\pm_1} \Psi \cdot \pr_{\pm_2} \Phi \|_{H^{s, b-1+\epsilon}} \lesa \| \Psi \|_{X^{s, b}_{\mp_1}} \| \Phi \|_{X^{s, b}_{\mp_2}},$$
where $\pm_1$ and $\pm_2$ are independent choices of $+$ and $-$, and $\Psi$ and $\Phi$ are functions taking values in $\mathfrak{g}\times \mathfrak{g}$. Observe that $$ \| \psi(- t, x) \|_{X^{s, b}_\pm } = \| \psi(t, x) \|_{X^{s, b}_\mp}, \qquad \| \psi(t, -x) \|_{X^{s, b}_\pm} = \| \psi(t, x) \|_{X^{s, b}_\pm}.$$
Similarly
        $$ \| \psi(- t, x) \|_{H^{s, b} } = \| \psi(t, x) \|_{H^{s, b}}, \qquad \| \psi(t, -x) \|_{H^{s, b}} = \| \psi(t, x) \|_{H^{s, b}}$$
and $\pr_\pm\big( f( - \cdot) \big)(x) = \pr_\mp f(-x)$. Furthermore a computation shows that $\beta \pr_\pm = \pr_\mp \beta$. Therefore, combining these observations, it suffices to prove
 \begin{equation}\label{mod mono est}   \| \pr_+ \Psi \cdot \pr_\pm \Phi \|_{H^{s, b-1+\epsilon}} \lesa \| \Psi \|_{X^{s, b}_+} \| \Phi \|_{X^{s, b}_\mp}
             \end{equation}

It is well known that nonlinear wave equations are only well behaved at low regularities if the nonlinear terms satisfy a null condition. The thesis of Czubak showed that the Monopole equation in the Coulomb gauge has null structure. Here we will show that the nonlinear term $\pr_+ \Psi \cdot \pr_\pm \Phi$ also has null structure in the sense that the worst interaction for parallel waves vanishes.  An easy computation shows that $ m_{\pm}(\xi)^T = m_{\pm}(\xi)$ and so
                $$ \widehat{\pr_+ \Psi \cdot \pr_\pm \Phi}(\xi) = \int_{\RR^2} m_\pm(\eta) m_+(\xi - \eta) \widehat{\Psi}(\xi-\eta) \cdot \widehat{\Phi}(\eta) d\eta.$$
Thus the symbol of $\pr_+ \Psi \cdot \pr_{\pm} \Phi$ is given by $m_\pm(\eta)m_+(\xi)$. The null structure is then contained in the following lemma.

\begin{lemma}\label{null struc}
We have the estimate
            $$ |m_+(\eta) m_\pm(\xi) | \lesa \theta(\xi, -\eta)$$
where $\theta(\xi, \eta)$ denotes the (positive) angle between $\xi$ and $\eta$.
\begin{proof}
The $(+, +)$ case follows from the computation
        \begin{align*}
           4 m_+(\eta) m_+(\xi) &= \left( I + \frac{1}{|\eta|} \alpha \cdot \eta \right) \left( I + \frac{1}{|\xi|} \alpha \cdot \xi \right) \\
                                  &= I + \frac{1}{|\eta| |\xi|} \begin{pmatrix}
                                                                         \eta_1 & \eta_2 \\
                                                                         \eta_2 & - \eta_1
                                                                \end{pmatrix}
                                                                                        \begin{pmatrix} \xi_1 & \xi_2 \\
                                                                                                          \xi_2 & - \xi_1 \end{pmatrix}
                                                                + \Big( \frac{\eta}{|\eta|} + \frac{\xi}{|\xi|} \Big) \cdot \alpha \\
                                  &=\Big( 1+ \frac{\xi \cdot \eta}{|\xi| |\eta|}\Big) I + \Big( \frac{\xi_2 \eta_1}{|\xi||\eta|} - \frac{\xi_1 \eta_2}{|\xi| |\eta|} \Big) \begin{pmatrix} 0 & 1 \\
                                                                              -1 & 0 \end{pmatrix}
                                                                              + \Big( \frac{\xi}{|\xi|} + \frac{\eta}{|\eta|} \Big) \cdot \alpha
        \end{align*}
together with the easy estimates $\Big( 1+ \frac{\xi \cdot \eta}{|\xi| |\eta|}\Big) \lesa \theta(\xi, -\eta)$, $\Big( \frac{\xi_2 \eta_1}{|\xi||\eta|} - \frac{\xi_1 \eta_2}{|\xi| |\eta|} \Big) \lesa \theta(\xi, -\eta)$, and $\Big( \frac{\xi}{|\xi|} + \frac{\eta}{|\eta|} \Big)\lesa \theta(\xi, -\eta)$. If we now note that $m_-(\eta) = m_+(-\eta)$ we obtain the $(+, -)$ case by replacing $\eta$ with $-\eta$ in the previous computation.

\end{proof}
\end{lemma}

Define $Q_{\pm}(\psi, \phi)$ by
            $$\widehat{ Q_\pm(\psi, \phi)}(\xi) = \int_{\RR^2} \theta(\xi-\eta, \pm \eta) \widehat{\psi}(\xi - \eta) \widehat{\phi}(\eta) d\eta.$$
Then by Lemma \ref{null struc} we have reduced the proof of Theorem \ref{local} to proving
            $$ \| Q_{\pm}(\psi, \phi) \|_{H^{s, b-1+\epsilon}} \lesa \| \psi \|_{X^{s, b}_+} \| \phi \|_{X^{s, b}_\pm}. $$
This estimate is essentially well known and follows from the work of Klainerman-Selberg \cite{Klainerman2002}, Foschi-Klainerman \cite{Foschi2000}, using ideas from \cite{D'Ancona2007}. However as we could not find this  inequality explicitly stated in the literature, we will include a proof  in the next section.  We note that the standard null form estimates for the wave equation in $\RR^{1+2}$ were proven by Zhou \cite{Zhou1997}. The origin of these types of estimates is the seminal paper of Klainerman-Machedon \cite{Klainerman1993}.

\section{Null-Form estimates}

Here we prove the following estimate.
\begin{theorem}\label{null form est}
Let $s>\frac{1}{4}$. Then there exists $b>\frac{3}{4}$ and $\epsilon>0$ with $b+\epsilon<1$ such that
           \begin{equation}\label{null form est main} \| Q_{\pm}(\psi, \phi) \|_{H^{s, b-1+\epsilon}} \lesa \| \psi \|_{X^{s, b}_+} \| \phi \|_{X^{s, b}_\pm}. \end{equation}
\end{theorem}
Note that this completes the proof of Theorem \ref{local}. To prove Theorem \ref{null form est} we need to introduce some notation. Let
        $$r_+ = |\xi - \eta| + |\eta| - |\xi|, \qquad \qquad r_- = |\xi| - \big| |\xi - \eta| - |\eta| \big|,$$
and define the bilinear operator $S^\alpha_\pm(\psi, \phi)$ by
        $$ \widehat{ S^\alpha_\pm(\psi, \phi) }(\xi) = \int_{\RR^2} r_\pm^\alpha \widehat{\psi}(\xi - \eta) \widehat{\phi}(\eta) d\eta.$$
Moreover define the Fourier multipliers $D^s$,  $\Lambda^s$, and $\Omega_\pm^b$ by
        $$ \widehat{D^s\psi}(\xi) = |\xi|^s \widehat{\psi}(\xi), \qquad \qquad \widehat{\Lambda^s \psi }(\xi) = \al \xi \ar^s \widehat{\psi}(\xi), \qquad \qquad \widehat{\Omega_\pm^b\psi}(\tau, \xi) = \al \tau \mp |\xi| \ar^b \widehat{\psi}(\xi).$$
Then we have the following estimate, which follows from \cite{Foschi2000} and is the analogue of Theorem 3.5 in \cite{Klainerman2002} for the $X^{s, b}_\pm$ spaces.

\begin{theorem}\label{homogeneous null}
Let $s, \alpha, s_1, s_2 \in \RR$ and $b'>\frac{1}{2}$. Then the estimate
        \begin{equation}\label{homogeneous null eqn1} \| D^s S^\alpha_\pm(\psi, \phi) \|_{L^2_{t, x}} \lesa \| D^{s_1} \psi \|_{X^{0, b'}_+} \| D^{s_2} \phi \|_{X^{0, b'}_\pm} \end{equation}
holds provided
        \begin{align*}
            s+\alpha &= s_1 + s_2 - \frac{1}{2} \\
            \alpha &\g \frac{1}{4} \\
            s_i &\les \alpha + \frac{1}{2}\\
            s_1+s_2 &\g \frac{1}{2}\\
            s&>\frac{-1}{2}
        \end{align*}
and $(s_i, \alpha) \neq (\frac{3}{4}, \frac{1}{4})$, $(s_1+s_2, \alpha) \neq (\frac{1}{2}, \frac{1}{4})$.
\begin{proof}
The hard work is contained in the result of Foschi-Klainerman \cite{Foschi2000} where the following estimate is proven
        $$ \big\| D^s D_-^\alpha (e^{it |\nabla| } f \, e^{\pm i t |\nabla| } g) \big\|_{L^2_{t, x}(\RR^{1+2})} \lesa \| D^{s_1} f\|_{L^2(\RR^2)} \| D^{s_2} g \|_{L^2(\RR^2)}$$
under the above conditions on the exponents $s, s_1, s_2, \alpha$ where $\widetilde{D_-^\alpha \psi }(\tau, \xi)= \big| |\tau| - |\xi| \big|^\alpha \widetilde{\psi}$. It is easy to see that
            $$ D_-^\alpha (e^{it |\nabla|} f \, e^{\pm i t |\nabla|} g) = S^\alpha_\pm (e^{it |\nabla|} f , e^{\pm i t |\nabla|} g).$$
Now since the operator $S^\alpha_\pm$ only acts on the $\xi$ variable, the expression on the lefthand side of (\ref{homogeneous null eqn1}) is invariant under multiplication by the modulations $e^{it \tau_0}$. Therefore  an application of the Transference principle\footnote{See for instance
Lemma 2.9 in \cite{Tao2006b}.} completes the proof.
\end{proof}
\end{theorem}

Theorem \ref{null form est} will now follow by using an argument from \cite{D'Ancona2007}.

\begin{proof}[Proof of Theorem \ref{null form est}]
We begin by noting that since the left and righthand sides of (\ref{null form est main})  only depend on the size of the Fourier transform of $\psi$ and $\phi$, we can use the triangle inequality to reduce to the case $\frac{1}{4} < s < \frac{1}{2}$. Choose $\epsilon>0$ and $b>\frac{3}{4}$ so that $s = b - \frac{1}{2} +\epsilon$. Note that $b + \epsilon < 1$.

We now deal with the low frequency case. Assume the product $\psi \phi$ has Fourier support contained in the set $\{ |\xi| <1\}$. Let $\rho \in C^\infty_0(\RR^2)$ with $\widehat{\rho}=1$ for $|\xi|<1$. Then
        \begin{equation}\label{conv}
          \psi \phi = \rho * (\psi \phi)
        \end{equation}
where the convolution is with respect to the $x$ variable. By discarding the smoothing multiplier $\al |\tau| - |\xi| \ar^{b-1+\epsilon}$, the null form $Q_\pm$, and using the assumption $\al \xi \ar \lesa 1$ together with (\ref{conv}), we have
    \begin{align*}
      \| Q_{\pm} (\psi, \phi) \|_{H^{s, b-1+\epsilon}} &\lesa \| \rho * (\psi \phi ) \|_{L^2_{t, x}} \\
                                                    &\lesa \| \psi \phi \|_{L^2_t L^1_x} \\
                                                    &\lesa \| \psi \|_{L^\infty_t L^2_x} \| \phi \|_{L^2_{t, x}} \\
                                                    &\lesa \| \psi \|_{X^{s, b}_+} \| \phi \|_{X^{s, b}_\pm}.
    \end{align*}
Therefore the low frequency case follows.

Since we may now assume $|\xi|>1$, it suffices to prove
            \begin{equation}\label{null form est eqn3} \| D^s Q_{\pm} ( \psi, \phi) \|_{H^{0, b-1+\epsilon} } \lesa \| \psi \|_{X^{s, b}_+} \| \phi \|_{X^{s, b}_\pm}. \end{equation}
 To this end we will need the following estimate on the symbol of $Q_\pm$,
        \begin{equation}\label{angle est} \theta^2(\xi - \eta, \eta) \approx\frac{|\xi - \eta| + |\eta|}{|\xi-\eta||\eta|} r_+, \qquad \qquad \theta^2(\xi-\eta, - \eta) \approx \frac{|\xi|}{|\xi- \eta||\eta|} r_-.\end{equation}
Note that these estimates  gives us a smoothing derivative $D^{-1}$ at the cost of a hyperbolic derivative $r_\pm$. To prove (\ref{angle est}) note that
        \begin{align*}
            ( |\eta| + |\xi - \eta| - |\xi|)(|\eta| + |\xi - \eta| + |\xi|) &= 2 \big(|\eta| |\xi - \eta| - \eta \cdot (\xi - \eta) \big) \\
                                                            &= 2 |\eta| |\xi - \eta| \big( 1 - \cos(\theta(\xi- \eta, \eta) )\big) \end{align*}
which proves the first estimate. For the second we have
 \begin{align*}    \big( |\xi| + \big| |\xi - \eta| - |\eta| \big| \big) \big( |\xi| - \big| |\xi - \eta| - |\eta| \big| \big)
                        &=2 \big( |\xi - \eta| |\eta| + \eta \cdot (\xi - \eta) \big) \\
                        &=2 |\xi - \eta| |\eta| \big( 1 - \cos( \theta(\xi - \eta, -\eta))\big) \end{align*}
and since $|\xi| \g \big| |\xi - \eta| - \eta \big|$ we have $|\xi | \approx |\xi| + \big| |\xi - \eta| - \eta \big| $ which gives the second estimate. We will also need the following estimate\footnote{The + case follows by writing
        $$ r_+ = (\tau - |\xi|) - ( \tau - \lambda - |\xi - \eta|) - (\lambda - |\eta|).$$
If $\tau>0$ the triangle inequality gives inequality while if $\tau<0$ then the term $(\tau - |\xi|)$ is less than zero and so can be discarded. The $-$ case follows from a similar computation after we note that
        $$ r_- \les \begin{cases} |\xi| + |\xi - \eta| - |\eta| \\
                                    |\xi| - |\xi - \eta| + |\eta|. \end{cases} $$}

            $$ r_\pm \les \big| |\tau| - |\xi|\big| + \big| |\tau - \lambda  - |\xi - \eta| \big| + \big| \lambda \mp |\eta| \big|$$
which leads to
                \begin{equation}\label{null form est eqn2}
                        r_\pm \lesa \al |\tau| - |\xi| \ar \al \tau - \lambda - |\xi - \eta| \ar \al \lambda \mp |\eta| \ar.
                \end{equation}

We are now ready to prove the $+$ case. Combining the estimates for $\theta$ and $r_+$ and assuming $|\eta| >|\xi - \eta|$ (as we may be symmetry) we have
    $$ \theta( \xi - \eta, \eta) \lesa \frac{r_+^{\frac{1}{2}}}{|\xi - \eta|^{\frac{1}{2}}} \lesa \frac{r_+^{b - \frac{1}{2} + \epsilon}}{|\xi - \eta|^{\frac{1}{2}}} \al |\tau| - |\xi| \ar^{1-b-\epsilon} \al \tau - \lambda - |\xi - \eta| \ar^{1-b-\epsilon} \al \lambda - |\eta| \ar^{1-b-\epsilon}.$$
and so
    $$ \| D^{s} Q_+(\psi, \phi) \|_{H^{0, b-1+\epsilon}} \lesa \Big\|  D^s S^{b - \frac{1}{2} + \epsilon}_+\big( \Omega_+^{1 - b- \epsilon}\psi, D^{-\frac{1}{2}} \Omega^{1-b-\epsilon}_+ \phi\big) \Big\|_{L^2_{t,x} }.$$
Therefore the $+$ case follows from Theorem \ref{homogeneous null} by taking\footnote{This is where we require the assumption $s>\frac{1}{4}$. As to apply Theorem \ref{homogeneous null} we need $\alpha>\frac{1}{4}$ and $s  + \alpha = s_1 + s_2 - \frac{1}{2}$ which implies $s = \alpha >\frac{1}{4}$. Note that if we could take $\alpha = 0$ then we would have local well-posedness for all $s>0$. However, heuristically speaking, since we have to assume $\alpha>\frac{1}{4}$ we can only use the null form $Q_\pm$ to cancel half the hyperbolic derivative $\al |\tau| - |\xi| \ar^{-\frac{1}{2}}$. See the related discussion after Theorem 3.3 in \cite{D'Ancona2007}.} $b' = 2b - 1 + \epsilon$, $s_1 = s$, $s_2 = s + \frac{1}{2}$, and $\alpha= b-\frac{1}{2} + \epsilon$. It is easy to check that the required conditions on $\alpha$, $s_1$, $s_2$, $s$, and $b'$ are satisfied.
To obtain the $-$ case we note that (\ref{angle est}) and (\ref{null form est eqn2}) give the estimate
            $$  \theta( \xi - \eta, -\eta) \lesa \frac{|\xi|^{\frac{1}{2}} r_-^{\frac{1}{2}}}{|\xi - \eta|^{\frac{1}{2}}|\eta|^{\frac{1}{2}}}
                        \lesa \frac{|\xi|^{\frac{1}{2}} r_-^{b - \frac{1}{2} + \epsilon}}{|\xi - \eta|^{\frac{1}{2}}|\eta|^{\frac{1}{2}}} \al |\tau| - |\xi| \ar^{1-b-\epsilon} \al \tau - \lambda - |\xi - \eta| \ar^{1-b-\epsilon} \al \lambda + |\eta| \ar^{1-b-\epsilon}.$$
Thus
    $$ \| D^s Q_-(\psi, \phi) \|_{H^{0, b-1+\epsilon}} \lesa \Big\|  S^{b - \frac{1}{2} + \epsilon}_-\big( D^{-\frac{1}{2}}\Omega_+^{1 - b- \epsilon}\psi, D^{-\frac{1}{2}} \Omega^{1-b-\epsilon}_- \phi\big) \Big\|_{L^2_{t, x} }$$
and so the required estimate follows from Theorem \ref{homogeneous null} by taking $b' = 2b - 1 + \epsilon$, $s_1 = s+\frac{1}{2}$, $s_2 = s + \frac{1}{2}$, and $\alpha= b-\frac{1}{2} + \epsilon$. Again it is easy to check that the required conditions are satisfied.

\end{proof}

\bibliographystyle{amsplain}
\bibliography{C:/math/Documents/Bibdata/DatabaseCurrent}

\end{document}